\documentclass[reqno, a4paper, 10pt]{amsart}
\usepackage{amssymb,url, color, mathrsfs}
\usepackage[colorlinks=true, bookmarks=true, pdfstartview=FitH, pagebackref=true, linktocpage=true, linkcolor = magenta, citecolor = blue]{hyperref}
\usepackage[short,nodayofweek]{datetime}
\usepackage[pagewise,running,mathlines,displaymath, switch]{lineno}

\theoremstyle{plain}
\numberwithin{equation}{section}
\newtheorem{theorem}{Theorem}[section]

\newtheorem{corollary}[theorem]{Corollary}
\newtheorem{proposition}[theorem]{Proposition}
\theoremstyle{remark}

\def\cfac#1{\ifmmode\setbox7\hbox{$\accent"5E#1$}\else\setbox7\hbox{\accent"5E#1}\penalty 10000\relax\fi\raise 1\ht7\hbox{\lower1.05ex\hbox to 1\wd7{\hss\accent"13\hss}}\penalty 10000\hskip-1\wd7\penalty 10000\box7 }

\def\aftil#1{\ifmmode\setbox7\hbox{$\accent"14#1$}\else\setbox7\hbox{\accent"14#1}\penalty 10000\relax\fi\raise 1\ht7\hbox{\lower1.1ex\hbox to 1\wd7{\hss\accent"7E\hss}}\penalty 10000\hskip-1\wd7\penalty 10000\box7 }

\def\Dbar{\leavevmode\lower.6ex\hbox to 0pt{\hskip-.03ex\accent"16\hss}D}

\author[Q.A. Ng\^o]{Qu\cfac oc Anh Ng\^o}

\address[Q.A. Ng\^o]{Department of Mathematics\\
College of Science, Vi\^{e}t Nam National University\\
H\`{a} N\^{o}i, Vi\^{e}t Nam.}
\email{\href{mailto: Q.A. Ng\^o <nqanh@vnu.edu.vn>}{nqanh@vnu.edu.vn}}
\email{\href{mailto: Q.A. Ng\^o <bookworm\_vn@yahoo.com>}{bookworm\_vn@yahoo.com}}

\author[V.H. Nguyen]{Van Hoang Nguyen}
\address[V.H. Nguyen]{Institut de Math\'ematiques de Toulouse\\
Universit\'e Paul Sabatier\\
31062 Toulouse c\'edex 09, France.}
\email{\href{mailto: V.H. Nguyen <van-hoang.nguyen@math.univ-toulouse.fr>}{van-hoang.nguyen@math.univ-toulouse.fr}}

\begin{document}

\allowdisplaybreaks

\setpagewiselinenumbers
\setlength\linenumbersep{100pt}
\modulolinenumbers[5]

\title[Poincar\'e-type inequalities in the hyperbolic space]
{Sharp constant for Poincar\'e-type inequalities in the hyperbolic space}

\begin{abstract} 
In this note, we establish a Poincar\'e-type inequality on the hyperbolic space $\mathbb H^n$, namely
\[
\|u\|_{p} \leqslant C(n,m,p) \|\nabla^m_g u\|_{p} 
\]
for any $u \in W^{m,p}(\mathbb H^n)$. We prove that the sharp constant $C(n,m,p)$ for the above inequality is
\[
C(n,m,p) =
\begin{cases}
\big( p  p'/(n-1)^2 \big)^{m/2}&\mbox{if $m$ is even},\\
(p/(n-1))\big( p  p'/(n-1)^2\big)^{(m-1)/2} &\mbox{if $m$ is odd},
\end{cases}
\]
with $p' = p/(p-1)$ and this sharp constant is never achieved in $W^{m,p}(\mathbb H^n)$. Our proofs rely on the symmetrization method extended to hyperbolic spaces.
\end{abstract}

\date{\bf \today \; at \, \currenttime}

\subjclass[2010]{26D10, 46E35, 31C12}

\keywords{Poincar\'e inequality, sharp constant, symmetrization method, hyperbolic space}

\maketitle


\section{Introduction}

Given a bounded, connected domain $\Omega \subset \mathbb R^n$ with Lipschitz boundary $\partial \Omega$, the classical Poincar\' e inequality with a sharp constant $C(p,\Omega)$ states that
\begin{equation}\label{eqPoincareInequality}
 \int_\Omega |u|^p dx \leqslant C(p,\Omega) \int_\Omega |\nabla u|^p dx
\end{equation}
for a ``suitable" function $u$ (usually in the Sobolev space $W^{1,p}(\Omega)$) with vanishing mean value on $\Omega$. Without assuming the vanishing mean value on $\Omega$, the classical Poincar\'e inequality reads as
\begin{equation}\label{eqPoincareInequalityWithAverage}
 \int_\Omega |u - \overline u|^p dx \leqslant C(p,\Omega) \int_\Omega |\nabla u|^p dx
\end{equation}
where $\overline u = (1/|\Omega|) \int_\Omega u dx$ denotes the mean value (or average) of $u$ over $\Omega$. Inequality \eqref{eqPoincareInequality} usually holds for $1 \leqslant p < +\infty$ under very general assumptions on $\Omega$, for example, it holds for domains satisfying the so-called ``segment property" or ``cone property"; see \cite{Agmon, liebloss2001}. An interesting question is that how the constant $C(p,\Omega)$ depends on the domain $\Omega$?

For $p=2$ and $n=3$, Steklov \cite{Ste96} showed that the constant $C(2,\Omega)$, when $\partial \Omega$ is piecewise smooth, must equal $1/\lambda_1$ where $\lambda_1$ is the first, non-zero eigenvalue of the following Neumann boundary condition problem
\[
\begin{cases}
-\Delta u = \lambda u & \text{ in } \Omega,\\
\partial_{\vec n} u = 0 & \text{ on }\partial \Omega.
\end{cases}
\]
Here $\vec n$ is the exterior unit normal to $\partial \Omega$. A similar result was also obtained by Steklov \cite{Ste97} for the Dirichlet boundary condition problem
\[
\begin{cases}
-\Delta u = \lambda u & \text{ in } \Omega,\\
u = 0 & \text{ on }\partial \Omega.
\end{cases}
\]
Based on these fundamental results, a few results for the sharp constant $C(2,\Omega)$ are known; for example, the sharp constant $C(2, B(0,1))$ for the unit ball in $\mathbb R^3$ is $1/j_{1,1}$ where $j_{1,1}$ is the first positive zero of the Bessel function $J_1$; see \cite[Subsection 2.2]{KN} and \cite{NR}. For a convex domain $\Omega \subset \mathbb R^n$ with diameter $d$, in a beautiful work by Payne and Weinberger \cite{PW}, the authors showed that \eqref{eqPoincareInequality} for $p=2$ can be obtained from weighted Poincar\'e inequalities in dimension one. As a consequence of this, they proved that $C(2,\Omega) = d/\pi$. A similar argument applied to the case $p=1$ gives $C(1,\Omega)=d/2$; see \cite{AD}. 

Poincar\'e inequalities for punctured domains was also studied in \cite{LSY}. For a general domain $\Omega$ and arbitrary $p$, determining the Poincar\'e constant $C(p,\Omega)$ is a hard task since the value $C(p,\Omega)$ depends on $p$ and the geometry of the domain $\Omega$.

In this note, we consider \eqref{eqPoincareInequality} for the hyperbolic space $\mathbb H^n$ with $n \geqslant 2$. The motivation of writing this note goes back to a recent high-order Poincar\'e-type inequality on $\mathbb H^n$ established by Karmakar and Sandeep in \cite{KS2016} and subsequently by a few works such as \cite{BG15, BGG15}; for interested readers, we refer to \cite{MS08, Tatatu01} for further details and related issues. To go further, let us briefly recall the definition of the space $\mathbb H^n$. 

The hyperbolic space $\mathbb H^n$ with $n\geqslant 2$ is a complete, simply connected Riemannian manifold having constant sectional curvature $-1$. There is a number of models for $\mathbb H^n$, however, the most important models are the half-space model, the ball model, and the hyperboloid (or Lorentz) model. In this note, we are interested in the ball model since this model is especially useful for questions involving rotational symmetry. 

Given $n \geqslant 2$, we denote by $B_n$ the open unit ball in $\mathbb R^n$. Clearly, $B_n$ can be endowed with the following Riemannian metric
\[
g(x) =\Big(\frac 2{1-|x|^2}\Big)^2 dx \otimes dx,
\]
which is then called the ball model of the hyperbolic space $\mathbb H^n$. In local coordinates, we have $g_{ij} = (2/(1-|x|^2))^2 \delta_{ij}$ and $g^{ij} = ((1-|x|^2)/2)^2 \delta^{ij}$. Clearly, one can think that $g$ is conformal to $dx^2$ with the conformal factor $\ln (2/(1-|x|^2))$. Then, it is well-known that volume element of $\mathbb H^n$ is given by 
\[
dV_g(x) = \Big(\frac2{1-|x|^2}\Big)^n dx,
\]
where $dx$ denotes the Lebesgue measure in $\mathbb R^n$. Let $d(0,x)$ denote the hyperbolic distance between the origin and the point $x$. In the ball model, it is well-known that 
$$d(0,x) = \ln \big( (1+|x|)/(1-|x|) \big)$$ 
for arbitrary $x\in B_n$. In this new context, we still use $\nabla$ and $\Delta$ to denote the Euclidean gradient and Laplacian as well as $\langle\cdot ,\cdot\rangle$ to denote the standard inner product in $\mathbb R^n$. Then, in terms of $\nabla$, $\Delta$, and $\langle\cdot ,\cdot\rangle$, with respect to the hyperbolic metric $g$,  the hyperbolic gradient $\nabla_g$, whose local coordinates is $g^{ij}\partial_j$, and the Laplacian-Beltrami operator $\Delta_g$, defined to be $\text{div}_g (\nabla \, \cdot )$, are given by
\[
\nabla_g = \Big(\frac{1-|x|^2}2\Big)^2 \nabla,\quad \Delta_g = \Big(\frac{1-|x|^2}2\Big)^2 \Delta + (n-2) \Big( \frac{1-|x|^2}2\Big)^2 \langle x,\nabla\rangle .
\]
For higher order derivatives, we shall adopt the following convention
\begin{equation*}
\nabla_g^m \cdot =
\begin{cases}
\Delta_g^{m/2} \cdot &\mbox{if $m$ is even,}\\
\nabla_g ( \Delta_g^{(m-1)/2} \cdot \,) &\mbox{if $m$ is odd.}
\end{cases}
\end{equation*}
Furthermore, for simplicity, we write $|\nabla_g^m \cdot |$ instead of $|\nabla_g^m \cdot |_g$ if no confusion occurs. Given a function $f$ on $\mathbb H^n$, we denote 
\[
\|f\|_p = \Big(\int_{\mathbb H^n} |f|^p dV_g\Big)^{1/p}
\]
and $\|\nabla_g^m f\|_p = \| |\nabla_g^m f|_g\|_p$, for each $1 \leqslant p < +\infty$ and integer $m\geqslant 1$. We use $W^{m,p}(\mathbb H^n)$ to denote the Sobolev space of order $m$ in $\mathbb H^n$. In \cite{KS2016}, the authors prove the following high-order Poincar\'e inequality
\begin{equation}\label{eq:HighOrderPoincare}
 \|\nabla^l_g u\|_2 \leqslant  \Big( \frac 2{n-1}\Big)^{m-l} \|\nabla^m_g u\|_2
\end{equation}
for all $u \in W^{m,2}(\mathbb H^n)$. In view of \eqref{eq:HighOrderPoincare}, one can ask: \textit{Whether the constant $(2/(n-1))^{m-l}$ is sharp and do we have a similar inequality for the $L^p$-norm?} We notice that it was claimed in \cite{BG15} that the constant $(2/(n-1))^{m-l}$ in \eqref{eq:HighOrderPoincare} is sharp; however, we have not found any proof of this yet. In this note, we seek for an answer to the above question.

In order to state our results, for each number $1 < p < +\infty$, let us denote the following constant
\begin{equation}\label{eqConstantC}
C(n,m,p) =
\begin{cases}
\big( p  p' /(n-1)^2 \big)^{m/2}&\mbox{if $m$ is even},\\
(p/(n-1))\big( p  p' /(n-1)^2\big)^{(m-1)/2} &\mbox{if $m$ is odd},
\end{cases}
\end{equation}
with $p' = p/(p-1)$. Clearly when $p=2$ and hence $p'=2$, we obtain $C(n,m,2)= \big(2/(n-1)\big)^m$. In this note, our first result is the following.

\begin{theorem}\label{thmMAIN}
Given $p > 1$, then the following inequality holds
\begin{equation}\label{eq:sharpPoincare}
\|u\|_{p} \leqslant C(n,m,p) \|\nabla^m_g u\|_{p}
\end{equation}
for $u \in W^{m,p}(\mathbb H^n)$. Moreover, the constant $C(n,m,p)$ is sharp and is never achieved in $W^{m,p}(\mathbb H^n)$.
\end{theorem}

As a consequence of Theorem \ref{thmMAIN}, we know that the sharp constant $C(3,1,2)$ is $1/2$ which is not $1/j_{1,1}$ as in the Euclidean case. Let us now go back to \eqref{eq:HighOrderPoincare}. By making use of Theorem \ref{thmMAIN} above, we obtain the following corollary, which generalizes \eqref{eq:HighOrderPoincare}.

\begin{corollary}\label{corMAIN}
Given $p > 1$, then the following inequality holds
\begin{equation}\label{eq:sharpPoincareLM}
\|\nabla^l_g u\|_{p} \leqslant C(n,m-l,p) \|\nabla^m_g u\|_{p}
\end{equation}
for $u \in W^{m,p}(\mathbb H^n)$. Moreover, the constant $C(n,m-l,p)$ is sharp and is never achieved in $W^{m,p}(\mathbb H^n)$.
\end{corollary}

As a special case of Corollary \eqref{corMAIN}, we conclude that the constant $(2/(n-1))^{m-l}=C(n,m-l,2)$ in \eqref{eq:HighOrderPoincare} is sharp. In view of the results in \cite{BG15}, it would be nice, since the sharp constant is never achieved, if there is an analogue of \eqref{eq:sharpPoincare} with reminders. We leave this topic for interested readers.


\section{Proofs}

In this section, we prove Theorem \ref{thmMAIN}. Our proof basically consists of two main parts. In the first part, we prove \eqref{eq:sharpPoincare}. Then in the second part, we show that the constant $C(n,m,p)$ is sharp. Now we start with the first part. 

\subsection{Proof of (\ref{eq:sharpPoincare})}

It is now known that the symmetrization argument works well in the setting of hyperbolic spaces. It is not only the key tool in the proof of several important inequalities such as the sharp Adams and Moser--Trudinger inequalities in $\mathbb H^n$ established in \cite{NgoNguyen16} but also a key tool in the present proof for the sharp Poincar\'e inequality. 

Let us now recall some facts about the rearrangement in the hyperbolic space $\mathbb H^n$. Let the function $f: \mathbb H^n\to \mathbb R$ be such that 
\[
\big |\{x\in \mathbb H^n\, :\, |f(x)| > t\} \big | = \int_{\{x\in \mathbb H^n\,:\, |f(x)|>t\}} dV_g < +\infty
\]
for every $t >0$. Its \textit{distribution function} is defined by
\[
\mu_f(t) = \big |\{x\in \mathbb H^n\, :\, |f(x)| > t\} \big |.
\]
Then its \textit{decreasing rearrangement} $f^*$ is defined by
\[
f^*(t) = \sup\{s > 0\, :\, \mu_f(s) > t\}.
\]
Since $f^*$ is non-increasing, the maximal function $f^{**}$ of $f^*$ is defined by
\[
f^{**}(s) = \frac 1s\int_0^s f^*(t) dt.
\]
It is well-known for any $p \in (1,+\infty)$ that
\begin{equation}\label{eq:Hardyinequality}
\Big(\int_0^{+\infty} f^{**}(s)^p ds\Big)^{1/p} \leqslant p' \Big(\int_0^{+\infty} f^*(s)^p ds\Big)^{1/p}.
\end{equation}
Now, we define $f^\sharp: \mathbb H^n \to \mathbb R$ by 
\[
f^\sharp(x) = f^*(|B(0,d(0,x))|),
\]
where $B(0,d(0,x))$ and $|B(0,d(0,x))|$ denote the ball centered at the origin $0$ with radius $d(0,x)$ in the hyperbolic space and its hyperbolic volume, respectively. Then for any continuous increasing function $\Phi: [0, +\infty) \to [0, +\infty)$ we have
\begin{equation}\label{eqIDENTITY}
\int_{\mathbb H^n} \Phi(|f|) dV_g = \int_{\mathbb H^n} \Phi(f^\sharp) dV_g.
\end{equation}
Moreover, the Polya--Szeg\"o principle concludes that
\[
\int_{\mathbb H^n} |\nabla_g \phi^\sharp|^p dV_g \leqslant \int_{\mathbb H^n} | \nabla_g \phi |^p dV_g
\]
for any function $\phi :\mathbb H^n \to \mathbb R$. Now we define a function $\Phi$ on $[0,+\infty)$ as follows
\[
\Phi(s) =  n \omega_n \int_0^s (\sinh r)^{n-1} dr, \quad s\geqslant 0.
\]
Clearly, $\Phi$ is a continuous and strictly increasing function from $[0,+\infty)$ to $[0,+\infty)$. Let $F$ denote the inverse function of $\Phi$. Then it is not hard to verify that $F$ is a continuous, strictly increasing function. Furthermore, it satisfies
\begin{equation}\label{eq:definitionofF}
s = n \omega_n \int_0^{F(s)} (\sinh r)^{n-1} dr
\end{equation}
for any $s\geqslant 0$. Depending on $m$ and for clarity, we divide this part into several small steps as follows.

\subsubsection{The case $m=1$} 

Let $u \in W^{1,p}(\mathbb H^n)$ be arbitrary. Upon normalization, if necessary, we can assume that $\|\nabla_g u\|_p =1$. Then by the Polya--Szeg\"o principle we know that $\|\nabla_g u^\sharp\|_p \leqslant 1$. Recall, by the definition, that 
\[
u^\sharp(x) = u^*(|B(0, d(0,x))|).
\] 
Let $\mu_u$ denote the distribution function of $u$. For $t >0$, let $\rho(t)$ denote the radius of the ball having the hyperbolic volume $\mu_u(t)$. Then, we have
\[
\mu_u(t) = \int_{B(0,\rho(t))} dV_g = n\omega_n \int_0^{\rho(t)} (\sinh s)^{n-1} ds.
\]   
From this and the definition of the function $F$, it is easy to check that
\[
\rho(t) = F(\mu_u(t)).
\]  
We now define
\begin{equation}\label{eq:vphi}
\varphi (s) = (n \omega_n)^{-p/(p-1)} \int_s^{+\infty} (\sinh F(t))^{-p(n-1)/(p-1)} dt
\end{equation}
and choose
\begin{equation}\label{eq:g}
g( \varphi (s)) = u^*(s).
\end{equation}
Clearly the function $\varphi$ is decreasing with
\[
-\varphi '(s) = \big( n \omega_n (\sinh F(s))^{n-1} \big)^{-p/(p-1)}.
\]
Concerning the function $g$, it is increasing and
\[
\int_0^{+\infty} (g'(s))^p ds = \int_{\mathbb H^n} |\nabla_g u^\sharp|^p dV_g \leqslant 1.
\]
Denote $\underline g = (g')^*$ the decreasing rearrangement of $g'$ on $(0,+\infty)$ and set
\begin{equation*}
f(s) = \int_0^{\varphi (s)} \underline g(t) dt.
\end{equation*}
We have $f(s) \geqslant u^*(s)$ and 
\[
\int_0^{+\infty} \underline g(s)^p ds =\int_0^{+\infty} (g'(s))^p ds \leqslant 1.
\] 
Via integration by parts, for any $0< a < b < +\infty$, we have 
\begin{equation}\label{eq:IBPconstant}
\begin{split}
\int_a^b f(s)^p ds = & -p\int_a^b s \varphi '(s) \underline g(\varphi (s)) f(s)^{p-1} ds  + bf(b)^p -a f(a)^p.
\end{split}
\end{equation}
Next we show that 
\begin{equation}\label{eq:limit0}
\lim_{a \searrow 0} a f(a)^p  = \lim_{b \nearrow +\infty} b f(b)^p =0.
\end{equation}
Indeed, for any $\varepsilon > 0$, there is $R > 0$ such that $\int_R^{+\infty} \underline g(s)^p ds < \varepsilon^p$, take $s_0$ such that $\varphi (s_0) =R$. Then, for $0 < a < s_0$, we have
\begin{align*}
f(a) &= \int_0^{\varphi (s_0)} \underline g(s) ds + \int_{\varphi (s_0)}^{\varphi (a)} \underline g(s) ds\\
&\leqslant \int_0^{\varphi (s_0)} \underline g(s) ds +\Big(\int_{\varphi (s_0)}^{\varphi (a)} \underline g(s)^p ds\Big)^{1/p} \big(\varphi (a) -\varphi (s_0) \big)^{(p-1)/p}\\
&\leqslant \int_0^{\varphi (s_0)} \underline g(s) ds +\varepsilon \big( \varphi (a) -\varphi (s_0) \big)^{(p-1)/p}.
\end{align*}
Since there holds $n\omega_n (\sinh F(s))^{n-1} \geqslant (n-1) s$ for all $s > 0$, we conclude that
\[\begin{split}
\varphi (a) -\varphi (s_0) \leqslant &\int_a^{s_0} \big( (n-1)s \big)^{-p/(p-1)}ds \\
= &(n-1)^{-p/(p-1)} (p-1) \big(a^{-1/(p-1)} -s_0^{-1/(p-1)} \big).
\end{split}\]
Therefore, we get
\begin{align*}
\limsup_{a \searrow 0} a f(a)^p \leqslant& \limsup_{a \searrow 0} a\Big(\int_0^{\varphi (s_0)} \underline g(s) ds +\varepsilon \big(\varphi (a) -\varphi (s_0) \big)^{(p-1)/p}\Big)^p\\
=&\limsup_{a \searrow 0} \Big[ a \varepsilon^p (\varphi (a) -\varphi (s_0))^{p-1} \Big]\\
\leqslant &(n-1)^{-p} (p-1)^{p-1} \varepsilon^p  \limsup_{a \searrow 0} a \big(a^{-1/(p-1)} -s_0^{-1/(p-1)} \big)^{p-1}\\
=&(n-1)^{-p}(p-1)^{p-1} \varepsilon^p.
\end{align*}
Since $\varepsilon$ is chosen arbitrarily, we get that $\limsup_{a \searrow 0} af(a)^p =0$ as claimed. The second limit in \eqref{eq:limit0} follows from the H\"older inequality. Indeed, first we notice that
\[
f(b)^p \leqslant \varphi (b)^{p-1}\int_0^{\varphi (b)} \underline g(s)^p ds .
\]
Observe that
\begin{equation}\label{eq:boundofvphi}
\begin{split}
\varphi (b) \leqslant &(n-1)^{-p/(p-1)}\int_b^{+\infty} s^{-p/(p-1)}ds \\
 \leqslant &(n-1)^{-p/(p-1)}(p-1) b^{- 1/(p-1)},
\end{split}
\end{equation}
which helps us to obtain
\[
b f(b)^p \leqslant (n-1)^{-p} (p-1)^{p-1} \int_0^{\varphi (b)} \underline g(s)^p ds.
\]
From this the conclusion follows since $\lim_{b\searrow 0} \int_0^{\varphi (b)} \underline g(s)^p ds =0$, which comes from the fact that $\varphi (b)$ tends to $0$ as $b$ tends to $0$. Thus, we have just established \eqref{eq:limit0}.

Let us now go back to \eqref{eq:IBPconstant}. Thanks to $\varphi ' \leqslant 0$, we can denote 
\[
h(s) = \underline g(\varphi (s)) (-\varphi '(s))^{1/p}.
\]
Clearly, $\int_0^{+\infty} h(s)^p ds \leqslant 1$. Making use of the H\"older inequality and  \eqref{eq:IBPconstant}, we can estimate $\int_a^b f(s)^p ds$ as follows
\begin{align*}
\int_a^b f(s)^p ds \leqslant  & p \Big(\int_a^b \big[ -\varphi '(s) s \underline g(\varphi (s))\big]^p ds\Big)^{1/p} \Big(\int_a^b f(s)^p ds\Big)^{(p-1)/p} \\
&+ b f(b)^p -a f(a)^p.
\end{align*}
First dividing both sides by $\big(\int_a^b f(s)^p ds\big)^{(p-1)/p}$, then letting $a \searrow 0$ and $b \nearrow +\infty$ and using \eqref{eq:limit0}, we obtain
\begin{equation}\label{eq:case1}
\Big(\int_0^{+\infty} f(s)^p ds\Big)^{1/p} \leqslant p \Big(\int_0^{+\infty}\big[ -\varphi '(s) s \underline g(\varphi (s))\big]^p ds\Big)^{1/p}.
\end{equation}
Note that the inequality $n\omega_n (\sinh F(s))^{n-1} > (n-1) s$, which holds for any $s>0$, and the definition of $\varphi$ imply that
\[
\big(-\varphi '(s) \big)^{(p-1)/p} s < (n-1)^{-1}
\]
for all $s >0$. Combining the latter inequality and \eqref{eq:case1}, we obtain
\[
\Big(\int_0^{+\infty} f(s)^p ds\Big)^{1/p} < \frac p{n-1} \Big(\int_0^{+\infty} h(s)^p ds\Big)^{1/p} \leqslant \frac p{n-1}.
\]
Since $u^* \leqslant f$, we have
\[
\Big(\int_{\mathbb H^n} |u|^p dV_g\Big)^{1/p} = \Big(\int_0^{+\infty} (u^*(s))^p ds\Big)^{1/p} \leqslant \Big(\int_0^{+\infty} f(s)^p ds\Big)^{1/p} < \frac p{n-1}
\]
for any function $u\in W^{1,p}(\mathbb H^n)$ with $\|\nabla_g u\|_p =1$. This proves \eqref{eq:sharpPoincare} for the case $m=1$ and also shows that the constant $C(n,1,p)$ is not achieved.

\subsubsection{The case $m=2$} 

Let $u\in W^{2,p}(\mathbb H^n)$ be such that $\|\Delta_g u\|_p = 1$. We denote $f = -\Delta_g u$. It was proved in \cite{NgoNguyen16} that
\[
u^*(s) \leqslant \int_s^{+\infty} \frac{t f^{**}(t)}{[n \omega_n (\sinh F(t))^{n-1}]^2}dt =: h(s)
\]
for all $s>0$. As in \eqref{eq:limit0} for the case $m=1$, we can easily prove that
\begin{equation}\label{eq:limit00}
\lim_{s \searrow 0} s h(s)^p = \lim_{s \nearrow +\infty} s h(s)^p = 0.
\end{equation}
For any $b> a> 0$, using integration by parts and the H\"older inequality, we arrive at
\begin{align*}
\int_a^b h(s)^p ds =& bh(b)^p -a h(a)^p + p \int_a^b h(s)^{p-1} \frac{s^2 f^{**}(s)}{[n \omega_n (\sinh F(s))^{n-1}]^2} ds \\
\leqslant &p \Big(\int_a^b h(s)^p ds\Big)^{(p-1)/p} \Big(\int_a^b \Big[\frac{s^2 f^{**}(s)}{[n \omega_n (\sinh F(s))^{n-1}]^2}\Big]^p ds\Big)^{1/p}\\
& +b h(b)^p -a h(a)^p.
\end{align*}
Dividing both sides by $\big(\int_a^b h(s)^p ds\big)^{1/p}$, letting $a \searrow 0$ and $b \nearrow +\infty$, and thanks to \eqref{eq:limit00}, we obtain
\begin{equation}\label{eq:case2}
\Big(\int_0^{+\infty} h(s)^p ds\Big)^{1/p} \leqslant p \Big(\int_0^{+\infty} \Big[\frac{s^2 f^{**}(s)}{[n \omega_n (\sinh F(s))^{n-1}]^2}\Big]^p ds\Big)^{1/p}.
\end{equation}
Using the inequality $n \omega_n (\sinh F(s))^{n-1} > (n-1) s$, \eqref{eq:Hardyinequality}, and \eqref{eq:case2}, we have
\[
\Big(\int_0^{+\infty} h(s)^p ds\Big)^{1/p} < \frac{pp'}{(n-1)^2} \Big(\int_0^{+\infty} f^*(s)^p ds\Big)^{1/p} \leqslant \frac{pp'}{(n-1)^2}.
\]
Since $u^* \leqslant h$, we then obtain
\[
\Big(\int_{\mathbb H^n} |u|^p dV_g\Big)^{1/p} = \Big(\int_0^{+\infty} (u^*(s))^p ds\Big)^{1/p} \leqslant \Big(\int_0^{+\infty} h(s)^p ds\Big)^{1/p} < \frac {pp'}{(n-1)^2}.
\]
Since the function function $u\in W^{2,p}(\mathbb H^n)$ with $\|\Delta_g u\|_p =1$ is arbitrary, this proves \eqref{eq:sharpPoincare} for the case $m=2$. In addition, this also shows that the constant $C(n,2,p)$ is not achieved.

\subsubsection{The case $m >2$}

In this scenario, we have two possible cases:

\noindent\textbf{Case 1}. Suppose that $m=2k$ is even. Clearly, this case follows from the case $m=2$ by repeating $k$ times as follows
\[\begin{split}
\|u\|_p \leqslant  \frac{p  p'}{(n-1)^2} \|\Delta_g u\|_p \leqslant & \Big(\frac{p  p'}{(n-1)^2} \Big)^2 \|\Delta_g^2 u\|_p  \\
\leqslant & \cdots \leqslant \Big(\frac{p  p'}{(n-1)^2} \Big)^k \|\Delta_g^k u\|_p.
\end{split}\] 

\noindent\textbf{Case 2}. Suppose that $m=2k+1$ is odd. This case can also be derived from the cases $m=1$ and $m=2$ as the following
\[\begin{split}
\|u\|_p \leqslant \frac p{n-1} \|\nabla_g u\|_p \leqslant & \frac p{n-1} \frac{p  p'}{(n-1)^2}\|\nabla_g (\Delta_g u)\|_p \\
 \leqslant & \cdots \leqslant  \frac p{n-1} \Big(\frac{p  p'}{(n-1)^2} \Big)^k  \|\nabla_g (\Delta_g^k u)\|_p.
\end{split}\] 

Let us now address the fact that the constant $C(n,m,p)$ cannot be achieved in $W^{m,p}(\mathbb H^n) \backslash \{ 0 \}$ for $m >2$; however this is easy and straightforward. Once we can prove this for $m=1,2$ with arbitrary $p$ as in the previous parts, we can easily deduce our statement for all $m \geqslant 3$ since
\[
C(n,m,p) =
\begin{cases}
C(n,2,p)^{m/2}&\mbox{if $m$ is even},\\
C(n,1,p) C(n,2,p)^{(m-1)/2} &\mbox{if $m$ is odd},
\end{cases}
\]
thanks to \eqref{eqConstantC}.

Before moving to the next stage of the proof, we note that by using the relation $\nabla_g (u^{p/2})=(p/2) u^{p/2-1}\nabla_g u$, the H\"older inequality, and the well-known fact that $C(n,1,2)$ is not achieved, it is also possible and perhaps easier to see that the sharp constant $C(n,m,p)$ is not achieved if $p >2$. In our argument above, we introduce a new idea, which crucially depends on \eqref{eq:limit0} and \eqref{eq:limit00}, to obtain the same result for any $p>1$ regardless of $C(n,1,2)$.

We now move to the second part of the proof. We shall prove the sharpness of $C(n,m,p)$ given in \eqref{eqConstantC} in the next subsection.

\subsection{The sharpness of $C(n,m,p)$}

It remains to check the sharpness of the constant $C(n,m,p)$. To do this, we will construct a function $u$ in such a way that $\|\nabla^m_g u\|_p/ \|u\|_p$ approximates $C(n,m,p)^{-1}$. Observe from \eqref{eq:definitionofF} that 
$$n\omega_n (\sinh F(s))^{n-1}\geqslant (n-1) s$$ 
for any $s\geqslant 0$ and
\[
\lim_{s\to +\infty} \frac{n\omega_n (\sinh F(s))^{n-1}}{(n-1) s} = 1.
\]
Hence, for any $\varepsilon >0$, there is $s_0$ such that
\[
(n-1) s \leqslant n\omega_n (\sinh F(s))^{n-1} \leqslant (1+\varepsilon) (n-1) s
\]
for all $s \geqslant s_0$. For any $R > s_0$, let us construct a positive, continuous, non-increasing function $f_R$ on $[0,+\infty)$ given by
\begin{equation}\label{eq:FunctionF_R}
f_R(s) = 
\begin{cases}
s_0^{-1/p}&\mbox{if $s\in (0,s_0)$},\\
s^{-1/p} &\mbox{if $s\in [s_0,R)$},\\
R^{-1/p} \max\{2- s/R, 0\} &\mbox{if $s\geqslant R$}.
\end{cases}
\end{equation}
Then we define two sequences of functions $\{v_{R,i}\}_{i \geqslant 0}$, $\{g_{R,i}\}_{i \geqslant 1}$ as follows: 
\begin{itemize}
  \item[(i)] first we set $v_{R,0} = f_R$; 
  \item[(ii)] then in terms of $v_{R,i}$, we define $g_{R,i+1}$ as the maximal function of $v_{R,i}$, that is
\[
g_{R,i+1}(s) = \frac1s \int_0^s v_{R,i}(t) dt;
\]
  \item[(iii)] and finally in terms of $g_{R,i+1}$ we define $v_{R,i+1}$ as follows
\[
v_{R,i+1}(s) = \int_s^{+\infty}\frac{t g_{R,i+1}(t)}{(n\omega_n (\sinh F(t))^{n-1})^2} dt,
\]
for $i=0, 1,2,...$ 
\end{itemize}
Note that $v_{R,i}$ and $g_{R,i}$ are non-increasing functions. We can explicitly compute the function $g_{R,1}$ as follows: When $s < R$ we have
\[
g_{R,1}(s) =
\begin{cases}
s_0^{-1/p}&\mbox{if $s\in (0,s_0)$}\\
p's^{-1/p} -s_0^{1-1/p}/((p-1)s) &\mbox{if $s\in [s_0,R)$,}
\end{cases}
\]
while for $s\in [R,2R)$ we have
\[
g_{R,1}(s) = \bigg(\Big(p'-\frac 32\Big) R^{1-1/p} - \frac{s_0^{1-1/p}}{p-1}\bigg)\frac1s +2R^{-1/p} -\frac{R^{-1-1/p} s}2,
\]
and finally when $s \geqslant 2R$ we have
\[
g_{R,1}(s) = \bigg(p' R^{1-1/p} - \frac{s_0^{1-1/p}}{p-1}\bigg)\frac1s + \frac{R^{1-1/p}}{2s}.
\]
Note that
\[
\int_R^{+\infty} g_{R,1} (s)^p ds \leqslant C
\]
for some constant $C >0$ independent of $R$

In the sequel, we use $C$ to denote various constants which are independent of $R$ and whose values can change from line to line and even in one line if no confusion occurs. We will need the following result.

\begin{proposition}\label{decomposefunction}
For any $i\geqslant 1$, there exist functions $h_{R,i}$ and $w_{R,i}$ such that 
$$v_{R,i} = h_{R,i} + w_{R,i},$$ 
that
\[
\int_0^{+\infty} |w_{R,i}|^p ds \leqslant C
\] 
and that
\[
\frac1{(1+\varepsilon)^{2i}} \Big(\frac{pp'}{(n-1)^2}\Big)^i f_R \leqslant h_{R,i} \leqslant \Big(\frac{pp'}{(n-1)^2}\Big)^i f_R.
\]
\end{proposition}

\begin{proof}
Let us define the operator $T$ acting on functions $v$ on $[0,+\infty)$ by
\[
( Tv) (s) = \int_s^{+\infty} \frac{r}{(n\omega_n (\sinh F(r))^{n-1})^2} \Big(\frac1r \int_0^r v(t) dt\Big) dr.
\]
For simplicity, for each function $v$ on $[0,+\infty)$ we define an associated function $\overline{v}$ on $\mathbb H^n$ by 
\[
\overline{v}(x) = v(|B(0,d(0,x))|).
\]
With these notations, it is not hard to see that
\[
\|\o{w_{R,i}}\|_p = \Big(\int_0^{+\infty} |w_{R,i}(s)|^p ds\Big)^{1/p}
\]
for any $i\geqslant 1$ and 
$$-\Delta_g \o{Tw_{R,i}}(x) = \o{w_{R,i}}(x)$$ 
for any $x\in \mathbb H^n$. Hence, by the Poincar\'e inequality, we have
\[
\int_0^{+\infty} |Tw_{R,i}(s)|^p ds = \|\o{Tw_{R,i}}\|_p^p \leqslant C \|\o{w_{R,i}}\|_p^p = C \Big(\int_0^{+\infty} |w_{R,i}(s)|^p ds\Big)^{1/p}.
\]
Thus, using an induction argument, it is enough to prove this proposition for $i=1$. We will perform several explicit estimation for the function $v_{R,1}$. Note that for $s\geqslant s_0$ we have
\[
(n-1) s \leqslant n\omega_n (\sinh F(s))^{n-1} \leqslant (1+\varepsilon) (n-1) s.
\]

\noindent\textbf{Estimate of $v_{R,1}$ when $s \geqslant 2R$}. Clearly for $s \geqslant 2R$, we have
\begin{align*}
v_{R,1}(s)& \leqslant \frac1{(n-1)^2}\int_s^{+\infty} \frac{(p'+1/2)R^{1-1/p} -s_0^{1-1/p}/(p-1)}{t^2} dt\\
&=\frac1{(n-1)^2}\frac{(p'+1/2)R^{1-1/p} -s_0^{1-1/p}/(p-1)}{s}
\end{align*}
and similarly we have
\[
v_{R,1}(s) \geqslant \frac1{(1+\varepsilon)^2(n-1)^2}\frac{(p'+1/2)R^{1-1/p} -s_0^{1-1/p}/(p-1)}{s}.
\]
Thus an easy calculation shows that
\begin{equation}\label{eqEstimateVOut2R}
\int_{2R}^{+\infty} v_{R,1}(s)^p ds \leqslant C.
\end{equation}
 
\noindent\textbf{Estimate of $v_{R,1}$ when $R \leqslant s < 2R$}. For $s\in [R,2R)$, we first write
\[
v_{R,1}(s) = v_{R,1}(2R) + \int_s^{2R} \frac{t g_{R,1}(t)}{(n\omega_n (\sinh F(t))^{n-1})^2}dt.
\]
Then we can estimate
\[\begin{split}
v_{R,1}(2R) +& \frac1{(1+\varepsilon)^2(n-1)^2}\int_s^{2R} \frac{g_{R,1}(t)}t dt \\
&\leqslant v_{R,1}(s) \leqslant v_{R,1}(2R)+ \frac1{(n-1)^2}\int_s^{2R} \frac{g_{R,1}(t)}t dt.
\end{split}\]
Note that $v_{R,1}(2R)$ is equivalent to $R^{-1/p}$ and
\begin{align*}
\int_s^{2R} \frac{g_{R,1}(t)}t dt =&\Big(\Big(p'-\frac32\Big) R^{1-1/p} - \frac{s_0^{1-1/p}}{p-1}\Big) \Big(\frac1s-\frac1{2R}\Big)\\
& +2R^{-1/p} \ln \frac{2R}s -\frac{R^{-1-1/p} (2R-s)}2.
\end{align*}
This shows that
\begin{equation}\label{eqEstimateVInR2R}
\int_R^{2R} v_{R,1}(s)^p ds \leqslant C
\end{equation}
and that $v_{R,1}(R)$ is equivalent to $R^{-1/p}$. Combining the estimates \eqref{eqEstimateVOut2R} and \eqref{eqEstimateVInR2R} gives $\int_R^{+\infty} v_{R,1}(s)^p ds \leqslant C$. 

\noindent\textbf{Estimate of $v_{R,1}$ when $s_0 \leqslant s < R$}. For $s\in [s_0,R)$, we also write
\[
v_{R,1}(s) = v_{R,1}(R) + \int_s^{R} \frac{t g_{R,1}(t)}{(n\omega_n (\sinh F(t))^{n-1})^2}dt.
\]
Thus
\[\begin{split}
v_{R,1}(R)+ & \frac1{(1+\varepsilon)^2(n-1)^2}\int_s^{R} \frac{g_{R,1}(t)}t dt \\
&\leqslant v_{R,1}(s) \leqslant v_{R,1}(R)+ \frac1{(n-1)^2}\int_s^{R} \frac{g_{R,1}(t)}t dt.
\end{split}\]
A simple computation gives
\[
\int_s^{R} \frac{g_{R,1}(t)}t dt = pp'(s^{-1/p} -R^{-1/p}) -\frac{s_0^{1-1/p}}{p-1} \Big(\frac1s -\frac1R\Big),
\]
which implies that
\[
\int_{s_0}^R \Big|\int_s^{R} \frac{g_{R,1}(t)}t dt - \frac{pp'}{s^{1/p}}\Big|^p ds \leqslant C.
\]

\noindent\textbf{Estimate of $v_{R,1}$ when $s <s_0$}. For $s \in (0,s_0)$ we write
\[
v_{R,1}(s) = v_{R,1}(s_0) + \int_s^{s_0} \frac{t s_0^{-1/p}}{(n\omega_n (\sinh F(t))^{n-1})^2}dt,
\]
therefore
\[
|v_{R,1}(s)|\leqslant C(R^{-1/p} + s_0^{2/n -1/p}).
\]
Consequently, we can write $v_{R,1} = h_{R,1} + w_{R,1}$ with $\int_0^{+\infty} |w_{R,1}|^p ds \leqslant C$ for some constant $C$ independent of $R$ and 
\[
\frac1{(1+\varepsilon)^2} \frac{pp'}{(n-1)^2} f_R \leqslant h_{R,1} \leqslant \frac{pp'}{(n-1)^2} f_R.
\]
(The way to see this is as follows: Since $\int_R^{+\infty} v_{R,1}(s)^p ds \leqslant C$, we can choose 
$$h_{R,1} = pp' (n-1)^{-2} f_R$$ 
when $r \geqslant R$. When $r < s_0$, we choose the same function for $h_{R,1}$. When $s_0 \leqslant r < R$, we choose 
$$h_{R,1} (s) = pp'(n-1)^{-2}/s^{1/p}$$ 
with a remark that $f_R (s) = s^{-1/p}$ in this scenario.) This finishes our proof of the proposition.
\end{proof}

We are now in position to confirm the sharpness of $C(n,m,p)$. For clarity, we split our proof into several small steps.

\subsubsection{The sharpness of $C(n,1,p)$}

We set 
$$u_R(x) =f_R(|B(0,d(0,x))|).$$ 
It is not hard to see that $u_R \in W^{1,p}(\mathbb H^n)$. We also consider the function $k_R$ defined by
\[
k_R(\varphi (s)) \varphi '(s) = f_R'(s).
\]
To finish our proof, we shall compute $\|\nabla_g u_R\|_p/\|u_R\|_p$. Indeed, we use \eqref{eqIDENTITY} to get
\[
\int_{\mathbb H^n} u_R(x)^p dV_g = \int_0^{+\infty} f_R(s)^p ds = 1 + \ln R -\ln s_0 + \int_0^1 (1-s)^p ds.
\]
For the gradient term, we observe that
\begin{align*}
\int_{\mathbb H^n} |\nabla_g u_R(x)|^p dV_g =& \int_0^{+\infty} k_R(s)^p ds\\
=&-\int_0^{+\infty} k_R(\varphi (s))^p \varphi '(s) ds\\
=&\int_0^{+\infty} (f_R'(s))^p (-\varphi '(s))^{1-p} ds\\
\leqslant &\frac{(n-1)^p}{p^p}(1+\varepsilon)^p \int_{s_0}^R s^{-1} ds \\
&+(n-1)^p (1+\varepsilon)^p R^{-p-1}\int_R^{2R} s^p ds \\
=&\frac{(n-1)^p}{p^p}(1+\varepsilon)^p (\ln R -\ln s_0) \\
&+ (n-1)^p (1+\varepsilon)^p \int_0^1 (1+s)^p ds.
\end{align*}
Hence
\[
\inf_{u\in W_0^{1,p}(\mathbb H^n) \backslash \{0\}} \frac{\int_{\mathbb H^n} |\nabla_g u|^p dV_g}{\int_{\mathbb H^n} |u|^p dV_g} \leqslant \liminf_{R\to +\infty} \frac{\int_{\mathbb H^n} |\nabla_g u_R|^p dV_g}{\int_{\mathbb H^n} |u_R|^p dV_g} \leqslant \frac{(n-1)^p}{p^p}(1+\varepsilon)^p.
\]
Since $\varepsilon >0$ is arbitrary, we obtain 
\[
\inf_{u\in W_0^{1,p}(\mathbb H^n)\backslash \{0\}} \frac{\int_{\mathbb H^n} |\nabla_g u|^p dV_g}{\int_{\mathbb H^n} |u|^p dV_g} \leqslant \Big(\frac{n-1}{p}\Big)^p.
\]
Hence the preceding inequality becomes equality. This proves the sharpness of $C(n,1,p)$. Next, we move to a proof for the sharpness of $C(n,2,p)$.

\subsubsection{The sharpness of $C(n,2,p)$}

In this case, we set 
$$u_R(x) = v_{R,1}(|B(0,d(0,x))|),$$ 
then we have 
\[
-\Delta_g u_R(x) = f_R(|B(0,d(0,x))|).
\]
Again, we shall compute $\|\Delta_g u_R\|_p/\|u_R\|_p$. Using this fact and \eqref{eqIDENTITY}, we easily obtain
\begin{equation}\label{eqIntegralDeltaU}
\int_{\mathbb H^n} |\Delta_g u_R|^p dV_g = \int_0^{+\infty} f_R(s)^p ds = 1+ \ln (R/s_0) + \int_0^1 (1-s)^pds.
\end{equation}
By Proposition \ref{decomposefunction}, we have
\begin{align*}
\|u_R\|_{p} &= \Big(\int_0^{+\infty} v_{R,1}(s)^p ds\Big)^{1/p} \\
&\geqslant \Big(\int_0^{+\infty} h_{R,1}(s)^p ds\Big)^{1/p} - \Big(\int_0^{+\infty} |w_{R,1}|^p ds\Big)^{1/p}\\
&\geqslant \frac1{(1+\varepsilon)^2} \frac{pp'}{(n-1)^2}\Big(\int_0^{+\infty} f_R(s)^p ds\Big)^{1/p} -C\\
&= \frac1{(1+\varepsilon)^2} \frac{pp'}{(n-1)^2} \Big(1 + \ln \big( \frac R{s_0} \big) + \int_0^1(1-t)^p dt\Big)^{1/p} -C.
\end{align*}
Combing this estimate and \eqref{eqIntegralDeltaU} gives
\[
C(n,2,p) \geqslant \liminf_{R\to +\infty} \frac{\|u_R\|_p}{\|\Delta_g u_R\|_p} \geqslant \frac1{(1+\varepsilon)^2}\frac{pp'}{(n-1)^2}.
\]
Since $\varepsilon >0$ is arbitrary, we conclude that
\[
C(n,2,p) \geqslant \frac{pp'}{(n-1)^2}
\]
and this finishes our proof for the case $m=2$.

\subsubsection{The sharpness of $C(n,2k,p)$ with $k\geqslant 2$} 

In this case, we set 
\[
u_R(x) = v_{R,k}(|B(0,d(0,x))|).
\]
Then it is clear to see that 
$$(-\Delta_g)^k u_R(x) = f_R (|B(0,d(0,x))|).$$ 
By Proposition \ref{decomposefunction}, we can write $v_{R,k} = h_{R,k} + w_{R,k}$ with $\int_0^{+\infty} |w_{R,k}|^p ds \leqslant C$ and 
\[
\frac1{(1+\varepsilon)^{2k}} \Big(\frac{pp'}{(n-1)^2}\Big)^k f_R \leqslant h_{R,k} \leqslant \Big(\frac{pp'}{(n-1)^2}\Big)^k f_R.
\]
Using a similar argument as in proving the sharpness of $C(n,2,p)$, we obtain the sharpness of $C(n,2k,p)$.

\subsubsection{The sharpness of $C(n,2k+1,p)$ with $k\geqslant 1$} 

In the previous argument, we can find a function $u_R$ on $\mathbb H^n$ such that 
$$(-\Delta_g)^k u_R(x) = f_R(|B(0,d(0,x))|)$$
and that
\[
\|u_R\|_p \geqslant \frac1{(1+\varepsilon)^{2k}} \Big(\frac{pp'}{(n-1)^2}\Big)^k \Big(\int_0^{+\infty} f_R(s)^p ds\Big)^{1/p}- C.
\]
From the proof of the sharpness of $C(n,1,p)$, we know that
\[
\int_{\mathbb H^n} |\nabla_g (\Delta_g^k u_R) |^p dV_g\leqslant \Big( \frac{ n-1 }{p} \Big)^p (1+\varepsilon)^p \ln \big(\frac R{s_0}\big) + (n-1)^p (1+\varepsilon)^p \int_0^1(1-t)^p dt.
\]
Combining these two estimate implies the sharpness of $C(n,2k+1,p)$ as claimed.


\section*{Acknowledgments}

V.H.N would like to acknowledge the support of the CIMI postdoctoral research fellowship. The research of Q.A.N is funded by the VNU University of Science under project number TN.16.01 and the Vietnam National Foundation for Science and Technology Development (NAFOSTED) under grant number 101.02-2016.02.

\section*{A note added}

After announcing our work on arXiv, see \cite{NgoNguyen16b}, it has come to our attention that the sharpness of $C(n,1,p)$ can be realized by a different argument by considering the upper half space model for $\mathbb H^n$, see \cite{BAGG17}.



\begin{thebibliography}{ww999999}

\bibitem[AD04]{AD}
\textsc{G. Acosta, R.G. Dur\'an},
An optimal Poincar\'e inequality in $L^1$ for convex domains, 
\textit{Proc. Amer. Math. Soc.} \textbf{132} (2004), no. 1, pp. 195--202.

\bibitem[Agm65]{Agmon}
\textsc{S. Agmon}, 
\textit{Lectures on elliptic boundary value problems}, 
Van Nostrand Company, 1965

\bibitem[BG16]{BG15}
\textsc{E. Berchio, D. Ganguly},
Improved higher order Poincar\'e inequalities on the hyperbolic space via Hardy-type remainder terms, 
\textit{Comm. Pure. Appl. Anal.} \textit{15} (2016), pp. 1871--1892.

\bibitem[BGG17]{BGG15}
\textsc{E. Berchio, D. Ganguly, G. Grillo},
Sharp Poincar\'e--Hardy and Poincar\'e--Rellich inequalities on the hyperbolic space, 
\textit{J. Funct. Anal.} \textbf{272} (2017), pp. 1661--1703.

\bibitem[BAGG17]{BAGG17}
\textsc{E. Berchio, L. D'Ambrosio, D. Ganguly, G. Grillo},
Improved $L^p$-Poincar\'e inequalities on the hyperbolic space, 
\textit{Nonlinear Anal.} \textbf{157} (2017), pp. 146--166.

\bibitem[KS16]{KS2016}
\textsc{D. Karmakar, K. Sandeep}, 
Adams inequality on the hyperbolic space, 
\textit{J. Funct. Anal.} \textbf{270} (2016), pp. 1792--1817.

\bibitem[KN15]{KN}
\textsc{N. Kuznetsov, A. Nazarov},
Sharp constants in Poincar\'e, Steklov and related inequalities (a survey),
\textit{Mathematika} \textbf{61} (2015), pp. 328--344.

\bibitem[LL01]{liebloss2001}
\textsc{E.H. Lieb, M. Loss},
Analysis,
\textit{2nd ed. Graduate studies in Mathematics {\bf 14}, Providence, RI: American Mathematical Sociery, 2001}.

\bibitem[LSY03]{LSY}
\textsc{E.H. Lieb, R. Seiringer, J. Yngvason},
Poincar\'e inequalities in punctured domains,
\textit{Ann. of Math.} \textbf{158} (2003), pp. 1067--1080.

\bibitem[MS08]{MS08}
\textsc{G. Mancini, K. Sandeep}, 
On a semilinear elliptic equation in $\mathbb H^N$, 
\textit{Ann. Sc. Norm. Super. Pisa Cl. Sci.} (5) \textbf{7} (2008), pp. 635--671.

\bibitem[NR15]{NR}
\textsc{A. Nazarov, S.I. Repin},
Exact constants in Poincar\'e type inequalities for functions with zero mean boundary traces,
\textit{Math. Methods Appl. Sci.} \textbf{38} (2015), no. 15, pp. 3195--3207. 

\bibitem[NN16a]{NgoNguyen16}
\textsc{Q.A. Ng\^o, V.H. Nguyen},
Sharp Adams-Moser-Trudinger type inequalities in the hyperbolic space,
arXiv:1606.07094.

\bibitem[NN16b]{NgoNguyen16b}
\textsc{Q.A. Ng\^o, V.H. Nguyen},
Sharp constant for Poincar\'e-type inequalities in the hyperbolic space,
arXiv:1607.00154.

\bibitem[PW60]{PW}
\textsc{L.E. Payne, H.F. Weinberger},
An optimal Poincar\'e inequality for convex domains,
\textit{ Arch. Rational Mech. Anal.} \textbf{5} (1960), pp. 286--292.

\bibitem[Ste96]{Ste96}
\textsc{V.A. Steklov}, 
On expansion of a function into the series of harmonic functions,
\emph{Proceedings of Kharkov Mathematical Society}, Ser. 2, \textbf{5} (1896), pp. 60--73.

\bibitem[Ste97]{Ste97}
\textsc{V.A. Steklov}, 
On expansion of a function into the series of harmonic functions,
\emph{Proceedings of Kharkov Mathematical Society}, Ser. 2, \textbf{6} (1897), pp. 57--124.

\bibitem[Tat01]{Tatatu01}
\textsc{D. Tataru},
Strichartz estimates in the hyperbolic space and global existence for the semilinear wave equation,
\textit{Trans. Amer. Math. Soc.} \textbf{353} (2001), pp. 795--807.

\end{thebibliography}
\end{document}